\documentclass[11pt]{amsart}

\pdfoutput=1

\usepackage[text={420pt,660pt},centering]{geometry}

\usepackage{cancel}
\usepackage{esint,amssymb}
\usepackage{graphicx}
\usepackage{mathtools} 
\usepackage[colorlinks=true, pdfstartview=FitV, linkcolor=blue, citecolor=blue, urlcolor=blue,pagebackref=false]{hyperref}
\usepackage{microtype}

\usepackage{bm}
\usepackage{dsfont}
\usepackage{mathrsfs}
\usepackage{xcolor}

\usepackage[normalem]{ulem}

\newtheorem{proposition}{Proposition}
\newtheorem{theorem}[proposition]{Theorem}
\newtheorem{lemma}[proposition]{Lemma}

\theoremstyle{remark}
\newtheorem{remark}[proposition]{Remark}

\theoremstyle{definition}

\numberwithin{equation}{section}
\numberwithin{proposition}{section}
\numberwithin{figure}{section}
\numberwithin{table}{section}

\newcommand{\N}{\mathbb{N}}
\newcommand{\Q}{\mathbb{Q}}
\newcommand{\R}{\mathbb{R}}

\newcommand{\E}{\mathbb{E}}

\renewcommand{\S}{\mathbf{S}}
\newcommand{\C}{\mathcal{C}}

\newcommand{\eps}{\varepsilon}

\renewcommand{\le}{\leqslant}
\renewcommand{\ge}{\geqslant}
\renewcommand{\leq}{\leqslant}
\renewcommand{\geq}{\geqslant}

\renewcommand{\subset}{\subseteq}

\renewcommand{\tilde}{\widetilde}

\renewcommand{\hat}{\widehat}

\newcommand{\Ll}{\left}
\newcommand{\Rr}{\right}
\renewcommand{\d}{\mathrm{d}}

\DeclareMathOperator{\tr}{tr}
\DeclareMathOperator{\supp}{supp}

\newcommand{\la}{\left\langle}
\newcommand{\ra}{\right\rangle}

\newcommand{\pert}{{\mathrm{pert}}}
\newcommand{\cF}{\mathcal{F}}
\newcommand{\geqpsd}{\succeq} \newcommand{\leqpsd}{\preceq}

\begin{document}

\author{Hong-Bin Chen}
\address{Institut des Hautes \'Etudes Scientifiques, France}
\email{hbchen@ihes.fr}


\title[Self-overlap correction simplifies the Parisi formula]{Self-overlap correction simplifies the Parisi formula for vector spins}

\begin{abstract}
We propose a simpler approach to identifying the limit of free energy in a vector spin glass model by adding a self-overlap correction to the Hamiltonian. This avoids constraining the self-overlap and allows us to identify the limit with the classical Parisi formula, similar to the proof for scalar models with Ising spins.
For the upper bound, the correction cancels self-overlap terms in Guerra's interpolation.
For the lower bound, we add an extra perturbation term to make the self-overlap concentrate, a technique already used in \cite{mourrat2021nonconvex,mourrat2023free} to ensure the Ghirlanda--Guerra identities. We then remove the correction using a Hamilton--Jacobi equation technique, which yields a formula similar to that in \cite{pan.vec}. Additionally, we sketch a direct proof of the main result in \cite{mourrat2020extending}.
\end{abstract}

\maketitle

\section{Introduction}

In \cite{pan.potts,pan.vec}, the limit of the free energy of mean-field vector spin glasses has been identified. One key insight is to consider the free energy with constrained self-overlap. Drawing inspiration from the Hamilton--Jacobi equation approach to spin glasses \cite{mourrat2022parisi,mourrat2020extending,mourrat2021nonconvex,mourrat2023free}, we present a simpler alternative approach. Specifically, we introduce a self-overlap correction to the Hamiltonian to simplify the analysis.

First, we employ the same argument used in \cite{pan} for the scalar model with Ising spins, along with Panchenko's synchronization technique, to identify the limit of the corrected free energy with the classical form of the Parisi formula (see Theorem~\ref{t.main}). Then, we remove the correction using a simple Hamilton--Jacobi equation technique (see Theorem~\ref{t.remove}).
Our approach also simplifies the analysis in the scalar model with soft spins \cite{pan05} as a special case.

\subsection*{Setting and main results}

Fix $D\in \N$ and let $P_1$ be a probability measure on $\R^D$. We assume that $P_1$ is supported on the unit ball $\{v\in \R^D: |v|\leq 1\}$. For each $N\in\N$, we define $P_N = (P_1)^{\otimes N}$ and denote the $\R^{D\times N}$-valued spin configuration sampled from $P_N$ by $\sigma = (\sigma_{i,j})_{i\in\{1,\dots, D\},\, j\in \{1,\dots,N\}}$. 

Throughout, we denote by $a^\intercal$ the transpose of a matrix or vector $a$. For two matrices or vectors $a$ and $b$ with the same size, we write $a\cdot b = \sum_{i,j}a_{i,j}b_{i,j} = \tr(ab^\intercal)$ as the Frobenius inner product between them, which naturally induces a norm $|a| = \sqrt{a\cdot a}$. 
We denote by $\S^D\subset \R^{D\times D}$ the set of symmetric matrices and $\S^D_+ \subset \S^D$ the set of symmetric positive semi-definite matrices. For $a,b\in\R^{D\times D}$, we write $a\geqpsd b$ and $b\leqpsd a$ if $a\cdot c\geq b\cdot c$ for all $c\in \S^D_+$.

Fix a differentiable and locally Lipschitz function $\xi:\R^{D\times D}\to \R$ satisfying $\xi\geq 0$ on $\S^D_+$, $\xi(0)=0$, and $\xi(a)=\xi(a^\intercal)$.
Using the Frobenius inner product on $\R^{D\times D}$, we can define the gradient $\nabla \xi:\R^{D\times D}\to \R^{D\times D}$ of $\xi$. We make the following additional assumption on $\xi$:
\begin{align}\label{e.assump_xi}
    a,b\in\S^D_+ \text{ and } a\geqpsd b \quad \Longrightarrow\quad  \xi(a) \geq \xi(b)\text{ and } \nabla\xi(a)\geqpsd \nabla\xi(b).
\end{align}

We assume that, for each $N$, there exists a real-valued centered Gaussian process $(H_N(\sigma))_{\sigma \in \R^{D\times N}}$ with covariance
\begin{align*}
    \E H_N(\sigma)H_N(\sigma') = N\xi\Ll(\frac{\sigma\sigma'^\intercal}{N}\Rr).
\end{align*}
Examples of $\xi$ and $H_N(\sigma)$ satisfying the above requirements are presented in \cite[Section~6]{mourrat2023free}. In particular, the mixed $p$-spin model with vector spins considered in \cite{pan.vec} is covered.

The limit of the free energy
\begin{align}\label{e.og_free_energy}
    \frac{1}{N}\E\log\int \exp\Ll(H_N(\sigma)\Rr) \d P_N(\sigma)
\end{align}
as $N\to\infty$ has been identified with a variational formula known as the Parisi formula in many settings.
We focus on the modified free energy with self-overlap correction:
\begin{align*}
    F_N = \frac{1}{N}\E\log\int\exp\Ll(H_N(\sigma) - \frac{1}{2}N\xi\Ll(\frac{\sigma\sigma^\intercal}{N}\Rr)\Rr)\d P_N(\sigma).
\end{align*}
The additional term $\frac{1}{2}N\xi\Ll(\frac{\sigma\sigma^\intercal}{N}\Rr)$ is one half of the variance of $H_N(\sigma)$. We view this as a correction of the \textit{self-overlap} $\frac{1}{N}\sigma\sigma^\intercal$ from the Hamiltonian. In Guerra's replica symmetry breaking interpolation, this correction term leads to the cancellation of terms involving the self-overlap.
This correction term already appeared in Mourrat's work \cite{mourrat2022parisi,mourrat2021nonconvex,mourrat2023free}.

We describe the Parisi functional.
For convenience, we use continuous versions of the Ruelle probability cascade (RPC) \cite{ruelle1987mathematical}.
Let $\mathfrak{R}$ be the RPC with overlap distributed uniformly on $[0,1]$ (see \cite[Theorem~2.17]{pan}). More precisely, $\mathfrak{R}$ is a random probability measure on the unit sphere of a separable Hilbert space with inner product denoted by $\wedge$ such that $\alpha^1\wedge \alpha^2$ distributes uniformly on $[0,1]$ under $\E\mathfrak{R}^{\otimes 2}$ and $\Ll(\alpha^l\wedge\alpha^{l'}\Rr)_{l,l'\in\N}$ satisfies the Ghirlanda--Guerra identities where $\Ll(\alpha^l\Rr)_{l\in\N}$ are i.i.d.\ samples from $\mathfrak{R}^{\otimes\infty}$.

Let $\Pi$ be the collection of left-continuous function $\pi:[0,1]\to \S^D_+$ that is increasing in sense that $\pi(s) \geqpsd  \pi(s')$ if $s\geq s'$.
For each $\pi\in\Pi$, let $(w^\pi(\alpha))_{\alpha\in\supp\mathfrak{R}}$ be a centered $\R^D$-valued Gaussian process with covariance
\begin{align}\label{e.w^pi}
    \E w^\pi(\alpha)\Ll(w^\pi(\alpha')\Rr)^\intercal = \pi(\alpha\wedge\alpha'),
\end{align}
conditioned on $\mathfrak{R}$.
We refer to \cite[Section~4 and Remark~4.9]{HJ_critical_pts} for the construction and measurability of this process.
Notice that $\xi(a)=\xi(a^\intercal)$ implies that $\nabla \xi(a)\in\S^D$ at $a\in\S^D$ and ~\eqref{e.assump_xi} (together with~\eqref{e.psd_fact}) implies that $\nabla\xi(a)\in \S^D_+$ at $a\in\S^D_+$. Hence, $\nabla\xi\circ\pi\in\Pi$ for every $\pi\in\Pi$.
We also define $\theta:\R^{D\times D}\to \R$ by
\begin{align}\label{e.theta}
    \theta(a) = a \cdot \nabla \xi(a) - \xi (a).
\end{align}
Then, we define $\mathscr{P}: \Pi\times\S^D \to\R$ by
\begin{align}
    \mathscr{P}(\pi,x) =\E \log \iint \exp\Ll( w^{\nabla\xi\circ\pi}(\alpha)\cdot \sigma - \frac{1}{2}\nabla\xi\circ\pi(1)\cdot \sigma\sigma^\intercal +x\cdot \sigma\sigma^\intercal\Rr) \d P_1(\sigma) \d \mathfrak{R}(\alpha) \notag
    \\
    + \frac{1}{2}\int_0^1 \theta(\pi(s))\d s. \label{e.parisi_functional}
\end{align}
Notice the correction $\frac{1}{2}\nabla\xi\circ\pi(1)\cdot \sigma\sigma^\intercal$, which is exactly half of the variance of $w^{\nabla\xi\circ\pi}(\alpha)\cdot \sigma$. We set $\mathscr{P}(\pi) = \mathscr{P}(\pi,0)$ which has the form of the classical Parisi functional.
\begin{theorem}[Parisi formula]\label{t.main}
If $\xi$ is convex, then $\lim_{N\to\infty} F_N = \inf_{\pi\in\Pi}\mathscr{P}(\pi)$.
\end{theorem}

\noindent
We can remove the correction in $F_N$ and obtain the limit of \eqref{e.og_free_energy} following the procedure in \cite[Section~5]{mourrat2020extending}.
Let $\xi^*$ be the convex conjugate of $\xi$ on $\S^D_+$ defined by
\begin{align}\label{e.xi^*}
    \xi^*(y) = \sup_{x\in \S^D_+}\{x\cdot y - \xi(x)\}.
\end{align}
\begin{theorem}[Removing correction]\label{t.remove}
If $\xi$ is convex, then
\begin{align}
    \lim_{N\to\infty} \frac{1}{N}\E\log\int \exp \Ll(H_N(\sigma)\Rr) \d P_N(\sigma) & = \sup_{y\in\S^D_+}\inf_{\pi\in\Pi} \Ll\{\mathscr{P}(\pi,y) - \frac{1}{2}\xi^*(2y)\Rr\}\notag
    \\
    & = \sup_{z\in\S^D_+}\inf_{y\in\S^D_+,\,\pi\in\Pi}\Ll\{\mathscr{P}(\pi,y) -y\cdot z+\frac{1}{2}\xi(z)\Rr\}.\label{e.lim_original}
\end{align}
\end{theorem}

The same argument for Theorem~\ref{t.main} can be used to treat models enriched by an external field given by a RPC. We describe the enriched model and sketch the proof of the corresponding result, Theorem~\ref{t.enriched}, in Section~\ref{s.enriched}. 

The convexity of $\xi$ is used in the proof of the upper bound via Guerra's interpolation.
To weaken this assumption to convexity over $\S^D_+$, one needs Talagrand's positivity principle, which is not available in general vector spin models.
Alternatively, an upper bound can be obtained through the Hamilton--Jacobi equation approach \cite{mourrat2021nonconvex,mourrat2023free,chen2022hamilton,chen2022hamilton2}. Based on this, statements in Theorems~\ref{t.main} and~\ref{t.remove} are proved in \cite[Corollary~8.3 and Proposition~8.4]{HJ_critical_pts} under the weaker assumption that $\xi$ is convex over $\S^D_+$.

\begin{remark}
It will be evident from the proof of the lower bound in Section~\ref{s.lower_bd} that there exists a minimizer $\pi$ of $\inf_{\pi\in\Pi}\mathscr{P}(\pi)$ that satisfies $\pi(1)\in \mathcal{K}= \overline{\mathrm{conv}}\Ll\{\tau\tau^\intercal:\:\tau\in\supp P_1\Rr\}$. Hence, we can replace $\inf_{\pi\in\Pi}$ in both Theorems~\ref{t.main} and~\ref{t.remove} by $\inf_{\pi\in\Pi:\: \pi(1)\in\mathcal{K}}$.
\end{remark}

\subsection*{Related works}
The classical Parisi formula for the limit of the free energy in the Sherrington--Kirkpatrick (SK) model ($D=1$, $P_1$ uniform on $\{\pm1\}$, and $\xi$ quadratic) was proposed by Parisi in \cite{parisi79,parisi80} and later proven by Talagrand in \cite{Tpaper} building on the upper bound by Guerra in \cite{gue03}. Panchenko later extended the formula to SK models with soft spins \cite{pan05}, scalar mixed $p$-spin models \cite{panchenko2014parisi,pan} with Ising spins, multi-species models \cite{pan.multi}, and mixed $p$-spin models with vector spins \cite{pan.potts,pan.vec}. 
Recently, \cite{bates2023parisi} simplified the Parisi formula for balanced Potts spin glass.
After Mourrat's interpretation of the Parisi formula as the Hopf--Lax formula for a Hamilton--Jacobi equation \cite{mourrat2022parisi}, the formula was extended for enriched models \cite{mourrat2020extending}. 
For spherical spins, the Parisi formula was proven for the SK model \cite{tal.sph}, the mixed $p$-spin model \cite{chen2013aizenman}, and the multi-species model \cite{bates2022free}.
Since we have assumed that $P_N$ is a product measure, the most relevant works are \cite{pan05,pan.potts,pan.vec,mourrat2020extending}.

Let us explain the effect of the self-overlap correction.
In Guerra's interpolation for the upper bound, there are terms involving the self-overlap, which have the wrong sign.
If the self-overlap is constant (as in the Ising case), these terms cancel each other.
Otherwise, to tackle this issue, \cite{pan.potts,pan.vec} considered free energy with the self-overlap constrained in a small ball, and controlled the original free energy by these with varying constraints.
In Section~\ref{s.upper_bd}, we demonstrate that by correcting $F_N$, the self-overlap terms in the interpolation computation are eliminated (as shown in~\eqref{e.dphi(r)}). Consequently, we can establish the upper bound in the same way as for the Ising case.

In the cavity computation in \cite{pan.potts,pan.vec} for the lower bound, the constraint of the self-overlap disrupts the product structure of $P_N$ but enables the derivation of Ghirlanda--Guerra identities in the limit.
Here, avoiding the constraint, we preserve the product measure structure in the cavity computation and proceed in the same way as for the Ising case (Section~\ref{s.lower_bd}).
The only issue is that the usual perturbation term added to the Hamiltonian does not ensure the Ghirlanda--Guerra identities since it only controls non-self-overlaps. To resolve this, we include additional perturbation (the second sum in~\eqref{e.H^pert}) that forces the self-overlap to concentrate. This technique has been previously utilized in \cite{mourrat2021nonconvex,mourrat2023free}.

Since the overlap is matrix-valued, in the proof of the lower bound, we also need the synchronization technique developed by Panchenko based on the ultrametricity of the overlap proved in \cite{pan.aom}.

To recover the limit of the original free energy given by~\eqref{e.og_free_energy}, we add an external field parameterized by $(t,x)\in [0,\infty)\times \S^D$ into $F_N$, which is denoted by $\mathcal{F}_N(t,x)$ (see equation~\eqref{e.cF_N}). The external field takes the form of $tN\xi(\frac{\sigma\sigma^\intercal}{N})+x\cdot \sigma\sigma^\intercal$. In Section~\ref{s.remove}, we will show that $\mathcal{F}_N(t,x)$ asymptotically satisfies a simple Hamilton--Jacobi equation $\partial_t f - \xi(\nabla f)=0$. It is worth noting that $\mathcal{F}_N(\frac{1}{2},0)$ corresponds to the original free energy, and the limit of $\mathcal{F}_N(0,x)$ is given by Theorem~\ref{t.main}, with $P_N$ tilted by $e^{x\cdot \sigma\sigma^\intercal}$. Therefore, we can interpolate along the equation to deduce Theorem~\ref{t.remove} from Theorem~\ref{t.main}.

It should be noted that the Hamilton--Jacobi equation mentioned here is a finite-dimensional one, rather than the infinite-dimensional one in \cite{mourrat2022parisi,mourrat2021nonconvex,mourrat2023free,chen2022hamilton} associated with the enriched model discussed in Section~\ref{s.enriched}. The former is similar to the one related to the Curie-Weiss model described in \cite{mourrat2021hamilton}. 
This idea for removing the correction first appeared in \cite[Section~5]{mourrat2020extending}.

In Section~\ref{s.enriched}, we consider the enriched model that was previously investigated in \cite{mourrat2020extending}. In that work, the limit of the free energy was determined by first establishing the Parisi formula similar to the ones in \cite{pan.potts,pan.vec} without correction. The formula was then transformed into the predicted form in \cite{mourrat2022parisi}. Here, we will sketch a direct proof using the argument outlined above.

Our simpler approach is not limited to product measures for $P_N$, making it useful in more general cases. However, our assumption on $P_N$ simplifies the computation and enables a clearer presentation of the argument.

\subsection*{Comments on variational formulae}

The formula presented in Theorem~\ref{t.main} has the classical form of the Parisi formula for Ising spins, as seen in \cite[Section~3.1]{pan}. However, the functional $\mathscr{P}(\pi)$ differs from the classical Parisi formula by a constant, which is a consequence of $\sigma\sigma^\intercal = N$ in the Ising case.

In Theorem~\ref{t.remove}, the two representations are obtained by solving the Hamilton--Jacobi equation mentioned earlier. The first representation is derived using the Hopf--Lax formula, which is possible due to the convexity of $\xi$. The second representation comes from the Hopf formula, taking advantage of the convexity of $y\mapsto \inf_{\pi\in\Pi}\mathscr{P}(\pi,y)$. One can verify the equivalence between the two using the convexity of the two functions and the Fenchel--Moreau theorem.

The first representation in Theorem~\ref{t.remove} is a generalization of \cite[Corollary~1.3]{mourrat2020extending} to $D\geq1$. The second representation is very close to the formula obtained by Panchenko in \cite{pan.vec}. The functional $\mathcal{P}$ in \cite[(31)]{pan.vec} can be rewritten as
\begin{align*}
    \mathcal{P}(y,z,\pi) = \mathscr{P}\Ll(\pi,y+ \frac{1}{2}\nabla\xi(z)\Rr)- \Ll(y+\frac{1}{2}\nabla\xi(z)\Rr)\cdot z+\frac{1}{2}\xi(z)
\end{align*}
subject to restriction $\pi(1)=z$ (here $y,z$ correspond to $\lambda,D$ in~\cite{pan.vec}). We set $\Pi(z)=\{\pi\in\Pi:\pi(1)=z\}$ and let $\mathcal{D}$ be the convex hull of $\{\sigma\sigma^\intercal:\sigma\in\supp P_1\}$. Notice that $\mathcal{D}\subset\S^D_+$. By \cite[Theorem~1]{pan.vec}, the left-hand side of~\eqref{e.lim_original} is equal to
\begin{align*}
    \sup_{z\in\mathcal{D}}\inf_{y\in\S^D,\,\pi\in\Pi(z)}\mathcal{P}(y,z,\pi)
    = \sup_{z\in\mathcal{D}}\inf_{y\in\S^D,\,\pi\in\Pi(z)}\Ll\{\mathscr{P}(\pi,y) -y\cdot z+\frac{1}{2}\xi(z)\Rr\},
\end{align*}
which is similar to our second representation in~\eqref{e.lim_original}.
The equivalence of the two formulae will be directly verified in~\cite{chen2023on} where more information on the self-overlap is needed.

\subsection*{Organization of the paper}
We establish Theorem~\ref{t.main} by dividing the proof into three parts: the upper bound (Section~\ref{s.upper_bd}), the perturbation term (Section~\ref{s.pertrubation}), and the lower bound (Section~\ref{s.lower_bd}). Theorem~\ref{t.remove}, which removes the correction term, is proved in Section~\ref{s.remove}. 
We also provide a brief overview of the enriched model and sketch the proof of the corresponding result in Section~\ref{s.enriched}.

\subsection*{Acknowledgements}
The author would like to thank Jean-Christophe Mourrat for helpful discussions. This project has received funding from the European Research Council (ERC) under the European Union’s Horizon 2020 research and innovation programme (grant agreement No.\ 757296).

\section{Upper bound via Guerra's interpolation}\label{s.upper_bd}

We prove the upper bound in Theorem~\ref{t.main} using Guerra's interpolation in \cite{gue03}, where the convexity of $\xi$ is needed.
We often need the following basic fact: if $a\in\S^D$, then
\begin{align}\label{e.psd_fact}
    a\cdot b\geq 0,\,\forall b\in \S^D_+\quad\Longleftrightarrow\quad a\in\S^D_+.
\end{align}
For $l,l'\in\N$, we write $R^{l,l'} = \frac{1}{N}\sigma^l(\sigma^{l'})^\intercal$ and $Q^{l,l'}=\alpha^l\wedge \alpha^{l'}$. By our setting and assumptions, $R^{l,l'}$ is $\R^{D\times D}$-valued satisfying $|R^{l,l'}|\leq 1$, and $Q^{l,l'}\in[0,1]$ is real-valued. 

\begin{proposition}\label{p.upp_bd}
If $\xi$ is convex, then $\limsup_{N\to\infty} F_N\leq \inf_{\pi\in\Pi}\mathscr{P}(\pi)$.
\end{proposition}

\begin{proof}
Recall the definition of $\theta$ in \eqref{e.theta}.
Notice that~\eqref{e.assump_xi} and~\eqref{e.psd_fact} imply that $b\cdot \nabla \xi(b) \leq b\cdot\nabla \xi(a)$ for all $a, b\in \S^D_+$ satisfying $a\geqpsd b$. Using this and the convexity of $\xi$, we have that
\begin{align*}
    \theta(a) -\theta(b) \geq \xi(b)-\xi(a)-(b-a)\cdot\nabla \xi(a) \geq 0
\end{align*}
for all $a,b\in\S^D_+$ satisfying $a\geqpsd  b$. Also, due to $\xi(0)=0$, we have $\theta(0)=0$. Hence, fixing any $\pi\in\Pi$, we have that $\theta\circ\pi:[0,1]\to[0,\infty)$ is increasing. Let $(y(\alpha))_{\alpha\in\supp\mathfrak{R}}$ be a centered real-valued Gaussian process with covariance $\E y(\alpha)y(\alpha') = \theta(\pi(\alpha\wedge\alpha'))$.
For $i\in\N$, let $(w_i^{\nabla\xi\circ\pi}(\alpha))_{\alpha\in\supp\mathfrak{R}}$ be i.i.d.\ copies of $(w^{\nabla\xi\circ\pi}(\alpha))_{\alpha\in\supp\mathfrak{R}}$ conditioned on $\mathfrak{R}$. For $\sigma\in \R^{D\times N}$, we write $\sigma= (\sigma_i)_{i=1}^N \in (\R^D)^N$ as a tuple of its column vectors. For $r\in[0,1]$, we define the interpolating Hamiltonian
\begin{align}
    H^r_N(\sigma,\alpha) = \sqrt{r} H_N(\sigma) - \frac{r}{2}N\xi\Ll(\frac{\sigma\sigma^\intercal}{N}\Rr) + \sqrt{1-r}\sum_{i=1}^N w_i^{\nabla\xi\circ\pi}(\alpha)\cdot \sigma_{i} -\frac{1-r}{2}\nabla\xi\circ\pi(1)\cdot \sigma\sigma^\intercal \notag
    \\
    + \sqrt{rN}y(\alpha) - \frac{rN}{2}\theta(\pi(1)),\label{e.H^t_upper_bd}
\end{align}
and the associated interpolating free energy
\begin{align}\label{e.phi(r)}
    \varphi(r) = \frac{1}{N}\E\log\iint \exp\Ll(H^r_N(\sigma,\alpha)\Rr)\d P_N(\sigma)\d \mathfrak{R}(\alpha).
\end{align}
Denoting the Gibbs measure with Hamiltonian $H^r_N(\sigma,\alpha)$ by $\la \cdot \ra_r$, we can compute using the Gaussian integration by parts (see~\cite[Section~1.2]{pan}) that
\begin{align}\label{e.dphi(r)}
    \frac{\d}{\d r}\varphi(r) = -\frac{1}{2}\E\la\xi\Ll(R^{1,2}\Rr)-\nabla\xi\Ll(\pi\Ll(Q^{1,2}\Rr)\Rr) \cdot R^{1,2} + \theta\Ll(\pi\Ll(Q^{1,2}\Rr)\Rr)\ra_r.
\end{align}
The self-overlap corrections have canceled all terms involving $R^{1,1}$ and $Q^{1,1}$, which is the key to the proof.
The convexity of $\xi$ ensures $\xi(a) - \nabla\xi(b)\cdot a -\theta(b)\geq 0$ for any $a,b\in \R^{D\times D}$. Therefore, we have $\frac{\d}{\d r}\varphi(r)\leq 0$ and thus $\varphi(1)\leq \varphi(0)$.

Let us rewrite the right-hand side of~\eqref{e.parisi_functional} as $\mathscr{P}_0(\pi)+\mathscr{P}_1(\pi)$.
We have $\varphi(0) = \mathscr{P}_0(\pi)$
\begin{align*}
    \varphi(1) = F_N +\frac{1}{N} \E \log \int \exp\Ll(\sqrt{N}y(\alpha) - \frac{N}{2}\theta(\pi(1))\Rr)\d \mathfrak{R}(\alpha).
\end{align*}
We want to show that the second term coincides with $\mathscr{P}_1(\pi)$. Assume that $\pi=\sum_{l=1}^kq_l\mathds{1}_{[s_{l-1},s_l)} + q_k\mathds{1}_{\{1\}}$ for some $k\in\N$, $0=s_0<s_1<\dots<s_k=1$ and $0= q_0\leqpsd q_1\leqpsd\cdots \leqpsd q_k$. Using the standard computation of the RPC in the proof of \cite[Lemma~3.1]{pan}, we have
\begin{align}
    &\frac{1}{N}\E \log \int \exp\Ll(\sqrt{N}y(\alpha)\Rr) \d\mathfrak{R}(\alpha) = \frac{1}{2}\sum_{l=0}^{ k-1}s_l\left(\theta(q_{l+1})-\theta(q_l)\right) \notag
    \\
    &= -\frac{1}{2}\Ll(\sum_{l=1}^k (s_{l}-s_{l-1})\theta(q_l) - s_k\theta(q_k)+s_0\theta(q_0)\Rr)
     =-\frac{1}{2} \int_0^1\theta(\pi(s))\d s + \frac{1}{2} \theta(\pi(1)).\label{e.logexpy}
\end{align}
Hence, we have $\varphi(1) = F_N -\mathscr{P}_1(\pi)$. The case where $\pi$ is continuous can be treated by the standard approximation. In conclusion, $\varphi(1)\leq \varphi(0)$ implies that $F_N\leq \mathscr{P}(\pi)$. Since $\pi$ is arbitrary, we can obtain the desired upper bound.
\end{proof}

\section{Perturbation and the Ghirlanda--Guerra identities}\label{s.pertrubation}

We directly work with the Hamiltonian appearing in the cavity computation in Section~\ref{s.lower_bd}.
For each $N$, let $(\tilde H_N(\sigma))_{\sigma\in \R^{D\times N}}$ be a centered real-valued Gaussian process with covariance
\begin{align}\label{e.tilde_H_N}
    \E\tilde H_N(\sigma)\tilde H_N(\sigma')= (N+1)\xi\Ll(\frac{\sigma\sigma'^\intercal}{N+1}\Rr).
\end{align}
We describe the perturbation term that will ensure the Ghirlanda--Guerra identities on average in the limit. 
Let $(a_n)_{n\in\N}$ be an enumeration of $\{q\in \S^D_+ :|q|\leq 1\}\cap \Q^{D\times D}$. For each $N\in \N$ and $h\in \N^3$, let $(H^h_N(\sigma))_{\sigma\in \R^{D\times N}}$ be an independent Gaussian process with
\begin{align*}\E H^h_N(\sigma)H^h_N(\sigma') = N\Ll(a_{h_1}\cdot \Ll(\frac{\sigma\sigma'^\intercal}{N}\Rr)^{\odot h_2}\Rr)^{h_3}
\end{align*}
where $\odot$ is the Schur product of matrices (i.e.\ $(a\odot b)_{ij} = (a_{ij}b_{ij})$). The construction of $H^h_N(\sigma)$ is described at the beginning of \cite[Section~5]{pan.vec} and omitted here. Note that $H^h_N(\sigma)$ has the same order of $\tilde H_N(\sigma)$. For each $h$, fix a constant $c_h>0$ satisfying
\begin{align*}c_h^2 \sup_{b\in \R^{D\times D},\, |b|\leq 1}\Ll(a_{h_1}\cdot b^{\odot h_2}\Rr)^{h_3}\leq 2^{-2|h|},
\end{align*}
where $|h|=h_1+h_2+h_3$.
Let $\{e_i\}_{i=1}^{D(D+1)/2}$ be an orthonormal basis of $\S^{D}$ and $g_i(\sigma) = e_i\cdot \sigma\sigma^\intercal$ for $i\in \{1,\dots, D(D+1)/2\}$.
For perturbation parameter
\begin{align*}x = \Ll((x_h)_{h\in\N^3}, (x_i)_{i=1}^{D(D+1)/2} \Rr)\in [0,3]^{\N^3}\times [0,3]^\frac{D(D+1)}{2},
\end{align*}
we set
\begin{align}\label{e.H^pert}
    H_{N}^{\pert,x}(\sigma) = \sum_{h\in \N^3} x_h  c_h H^h_N(\sigma)+ \sum_{i=1}^{D(D+1)/2} x_ig_i(\sigma).
\end{align}
The second sum will ensure the concentration of the self-overlap, allowing us to get the Ghirlanda--Guerra identities.

For $l,l'\in\N$, recall $R^{l,l'}$ below~\eqref{e.psd_fact}; write $R^{l,l'}_h = (a_{h_1}\cdot (R^{l,l'})^{\odot h_2})^{h_3}$ for each $h$, and $R^{\leq n} = \Ll(R^{l,l'}\Rr)_{l,l'\leq n} $ for each $n\in\N$.
Let $\E_x$ be the expectation under which $x$ is an i.i.d.\ sequence of uniform random variables over $[1,2]$. Let $\la\cdot\ra_x$ be the Gibbs measure of $\sigma$ with Hamiltonian
\begin{align}\label{e.tildeH^x_N}
    \tilde H_N^x(\sigma)=\tilde H_N(\sigma) -\frac{1}{2}(N+1)\xi\Ll(\frac{\sigma\sigma^\intercal}{N+1}\Rr) + N^{-\frac{1}{16}} H^{\pert,x}_N(\sigma).
\end{align}
In the following, $\E$ only integrates the Gaussian randomness in $\tilde H_N(\sigma)$ and $H^{\pert,x}_N(\sigma)$.
\begin{proposition}\label{p.perturbation}
The following holds:
\begin{enumerate}
    \item \label{i.self-overlap_concent} $\lim_{N\to\infty} \E_x \E \la \Ll|R^{1,1} - \E \la R^{1,1}\ra_x\Rr|\ra_x =0$;
    \item \label{i.GG} $\lim_{N\to\infty} \E_x \Delta^x(f,n,h)=0$ for every integer $n\geq 2$, every $h\in \N^3$, and every bounded measurable function $f$ of $R^{\leq n}$ 
    where
    \begin{align}\label{e.Delta^x}
    &\Delta^x(f,n,h)\notag
    \\&= \Ll|\E \la f(R^{\leq n})R^{1,n+1}_h \ra_x - \frac{1}{n}\E\la f(R^{\leq n})\ra_x \E \la  R^{1,2}_h\ra_x - \frac{1}{n}\sum_{l=2}^n \E \la f(R^{\leq n}) R^{1,l}_h\ra_x\Rr|.
\end{align}
\end{enumerate}

\end{proposition}

The technique of adding $g_i(\sigma)$ to enforce the self-overlap to concentrate already appeared in \cite{mourrat2021nonconvex,mourrat2023free}.

We first prove the concentration of self-overlap and then the second part.
In the first part, $i,j$ are always indices in $\{1,\dots,D(D+1)/2\}$.

\begin{lemma}
There is a constant $C>0$ such that, for every $i$,
\begin{align*}
    \int_1^2\E \la\Ll|g_i(\sigma) - \E \la g_i(\sigma)\ra\Rr| \ra_x \d x_i \leq CN^{\frac{25}{32}}
\end{align*}
uniformly in $(x_h)_{h\in \N^3}$ and $(x_j)_{j\neq i}$.
\end{lemma}
\begin{proof}
Recall $\tilde H_N^x(\sigma)$ in~\eqref{e.tildeH^x_N}.
Fixing any $(x_h)_{h\in \N^3}$ and $(x_j)_{j\neq i}$, we set
\begin{align*}
    \varphi(x_i) = \log \int \exp \Ll(\tilde H_N^x(\sigma) \Rr)\d P_N(\sigma)
\end{align*}
and $\phi(x_i) =\E \varphi(x_i)$.
Then, we compute the derivatives:
\begin{gather*}
    \varphi'(x_i)  = N^{-\frac{1}{16}} \la g_i(\sigma)\ra_x, \qquad \varphi''(x_i) = N^{-\frac{1}{8}}  \la g_i(\sigma)^2 - \la g_i(\sigma)\ra^2_x \ra_x,
    \\
    \phi'(x_i)  = N^{-\frac{1}{16}} \E \la g_i(\sigma)\ra_x,\qquad 
    \phi''(x_i)  = N^{-\frac{1}{8}} \E \la g_i(\sigma)^2 - \la g_i(\sigma)\ra^2_x \ra_x.
\end{gather*}
Due to our assumption on the support of $P_N$, we have $|g_i(\sigma)|\leq N$ and thus
\begin{align}\label{e.|phi'|<}
    |\phi'(x_i)|\leq N^\frac{15}{16},\quad\forall x_i\in[0,3].
\end{align}
Integrating $\phi''(x_i)$ over $x_i\in[1,2]$ yields
\begin{align*}
    \int_1^2 \E \la g_i(\sigma)^2 - \la g_i(\sigma)\ra^2_x \ra_x \d x_i=N^\frac{1}{8}(\phi'(2)-\phi'(1)) \leq 2N^\frac{17}{16}.
\end{align*}

By the standard concentration argument (\cite[(4.15)]{mourrat2021nonconvex}), we have
\begin{align*}
    \sup_{x}\E |\varphi(x_i)-\E \varphi(x_i)| \leq CN^\frac{1}{2},
\end{align*}
for some constant $C$. 
Writing
\begin{align*}
    \delta = |\varphi(x_i+y_i) -\phi(x_i+y_i)| + |\varphi(x_i-y_i)-\phi(x_i-y_i)| + |\varphi(x_i)-\phi(x_i)|,
\end{align*}
we thus have $\int_1^2\delta \d x_i\leq 3CN^\frac{1}{2}$ for all $y_i\in[0,1]$. It is clear from our computation that $\varphi$ and $\phi$ are convex. By basic properties of convex functions stated in \cite[Lemma~3.2]{pan},
\begin{align*}
    \Ll|\varphi'(x_i) - \phi'(x_i)\Rr|\leq \phi'(x_i+y_i) - \phi'(x_i-y_i) + \frac{\delta}{y_i}.
\end{align*}
Using~\eqref{e.|phi'|<} and the mean value theorem, we get
\begin{align*}
    &\int_1^2 \Ll(\phi'(x_i+y_i)-\phi'(x_i-y_i)\Rr)\d x_i
    \\
    &= \phi(2+y_i) -\phi(2-y_i) - \phi(1+y_i) +\phi(1-y_i) \leq 4 N^\frac{15}{16} y_i.
\end{align*}
Therefore,
\begin{align*}
    \int_1^2 \E\Ll|\varphi'(x_i) - \phi'(x_i)\Rr| \d x_i \leq 4 N^\frac{15}{16} y_i + \frac{3CN^\frac{1}{2}}{y_i}.
\end{align*}
Inserting the formulae of $\varphi'$ and $\phi'$ and setting $y_i= N^{-\frac{7}{32}}$, we get the desired result. 
\end{proof}

\begin{proof}[Proof of~\eqref{i.self-overlap_concent}]
The above lemma implies that, uniformly in $(x_h)_{h\in \N^3}$ and $(x_j)_{j\neq i}$,
\begin{align*}
    \int_1^2 \E \la   \Ll|e_i\cdot\frac{\sigma\sigma^\intercal }{N} - \E \la e_i\cdot \frac{\sigma\sigma^\intercal }{N} \ra_x \Rr|\ra_x\d x_i \leq C N^{-\frac{7}{32}}
\end{align*}
for every $i$.
Since $\{e_i\}$ is an orthonormal basis, we can deduce Proposition~\ref{p.perturbation}~\eqref{i.self-overlap_concent}.
\end{proof}

\begin{proof}[Proof of~\eqref{i.GG}]
We can proceed in the standard way as in \cite[Theorem~3.3]{pan} (with $h$, $N^\frac{7}{16}$, $N^\frac{1}{2}$, $N^{-\frac{1}{2}}x_hc_h H^h_N(\sigma)$ substituted for $p$, $s$, $\nu_N(s)$, $2^{-p}x_pg_p(\sigma)$ therein) to get that for every $h$, there is a constant $C_h$ such that
\begin{align*}
    \int_1^2 \E \la \Ll| H^h_N -\E \la H^h_N\ra_x\Rr|\ra_x \d x_h \leq C_h N^\frac{3}{4}.
\end{align*}
Let $f = f(R^{\leq n})$ be bounded and measurable. Without loss of generality, we assume $\|f\|_\infty\leq 1$. Then,
\begin{align*}
    \Ll|\E\la f H^h_N\ra_x - \E \la f\ra_x \E \la H^h_N\ra_x\Rr| \leq\E \la \Ll|H^h_N - \E \la H^h_N\ra_x \Rr|\ra_x.
\end{align*}
The Gaussian integration by parts gives
\begin{align*}
    &\E\la f H^h_N\ra_x - \E \la f\ra_x \E \la H^h_N\ra_x 
    \\
    & = N^\frac{15}{16}x_hc_h\Ll(\E \la f\Ll(\sum_{l=1}^n R^{1,l}_h - n R^{1,n+1}_h\Rr)\ra_x + \E\la f\ra_x \E \la  R^{1,2}_h -R^{1,1}_h\ra_x\Rr).
\end{align*}
The above three displays yield
\begin{align*}
    &\E_x\Ll| n\E \la f R^{1,n+1}_h \ra_x - \E \la f\ra_x \E \la R^{1,2}\ra_x  -  \sum_{l=2}^n \E \la fR^{1,l}_h\ra_x\Rr|
    \\
    &\leq \E_x \E \la \Ll|R^{1,1}_h - \E \la R^{1,1}_h\ra_x  \Rr|\ra_x  +C_h N^{-\frac{3}{16}} .
\end{align*}
Since $|R^{1,1}|\leq 1$ holds due to our assumption, there is a constant $C'_h$ such that $|R^{1,1}_h - R^{2,2}_h|\leq C'_h |R^{1,1} - R^{2,2}|$ and thus
\begin{align*}
     \E \la \Ll|R^{1,1}_h - \E \la R^{1,1}_h\ra_x  \Rr|\ra_x \leq \E \la \Ll| R^{1,1}_h - R^{2,2}_h \Rr|\ra_x  \leq C'_h \E \la \Ll| R^{1,1} - R^{2,2} \Rr|\ra_x 
     \\
     \leq 2C'_h \E \la \Ll| R^{1,1}  -  \E \la R^{1,1}\ra_x \Rr|\ra_x.
\end{align*}
By Proposition~\ref{p.perturbation}~\eqref{i.self-overlap_concent}, the last term averaged by $\E_x$ vanishes as $N\to\infty$.
Inserting this into the previous display, we get Proposition~\ref{p.perturbation}~\eqref{i.GG}, which completes the proof.
\end{proof}

\section{Lower bound via the Aizenman--Sims--Starr scheme}
\label{s.lower_bd}

We show the lower bound in Theorem~\ref{t.main} using the Aizenman--Sims--Starr scheme from \cite{aizenman2003extended}. This part does not need the convexity of $\xi$.
\begin{proposition}\label{p.lower_bd}
It holds that $\liminf_{N\to\infty} F_N \geq \inf_{\pi\in\Pi}\mathscr{P}(\pi)$.
\end{proposition}
\noindent Theorem~\ref{t.main} follows from this and Proposition~\ref{p.upp_bd}.

Recall the perturbation defined in~\eqref{e.H^pert}. For each perturbation parameter $x$ and each $N$, we define
\begin{align*}
    F^x_N = \frac{1}{N}\E\log\int\exp\Ll(H_N(\sigma) -\frac{1}{2}N\xi\Ll(\frac{\sigma\sigma^\intercal}{N}\Rr)+ N^{-\frac{1}{16}} H^{\pert,x}_N(\sigma)\Rr)\d P_N(\sigma)
\end{align*}
as the perturbation of $F_N$.
Then, we describe the Gibbs measure that will appear in the cavity computation.
Recall the process $\tilde H_N(\sigma)$ in~\eqref{e.tilde_H_N} and the Gibbs measure $\la \cdot \ra_x$ with Hamiltonian $\tilde H_N^x(\sigma)$ in~\eqref{e.tildeH^x_N}.
We also need two more independent Gaussian processes. Let $(Z(\sigma))_{\sigma\in \R^{D\times N}}$ and $(Y(\sigma))_{\sigma \in \R^{D\times N}}$ be centered $\R^{D}$-valued and real-valued Gaussian processes with covariances $\E Z(\sigma)Z(\sigma')^\intercal = \nabla\xi\Ll(\frac{\sigma\sigma'^\intercal}{N}\Rr)$ and $\E Y(\sigma)Y(\sigma') =\theta\Ll(\frac{\sigma\sigma'^\intercal}{N}\Rr)$. The cavity computation evaluates the difference $(N+1)F_{N+1}^x -NF_N^x$. We write the spins in $F_{N+1}^x$ as $(\sigma,\tau)\in \R^{D\times N}\times \R^{D}$. Defining
\begin{align}
    A_N(x) 
    = \E \log\la \int \exp\Ll(Z(\sigma)\cdot \tau - \frac{1}{2}\nabla\xi\Ll(\frac{\sigma\sigma^\intercal}{N}\Rr)\cdot \tau\tau^\intercal\Rr)  \d P_1(\tau)\ra_x  \notag
    \\
    - \E \log \la \exp\Ll(Y(\sigma)
    - \frac{1}{2}\theta\Ll(\frac{\sigma\sigma^\intercal}{N}\Rr)\Rr) \ra_x \label{e.A_N(x)}
\end{align}
and using the standard computation described in the proof of \cite[Theorem~3.6]{pan}, we can obtain the Aizenman--Sims--Starr representation stated below.

\begin{lemma}\label{l.cavity}
Uniformly in $x$, $(N+1)F_{N+1}^x - NF_N^x = A_N(x) + o(1)$ as $N\to\infty$.
\end{lemma}

\noindent
We also need a result on approximating the Parisi-type functional using finitely many entries from the overlap array. The following is a straightforward adaption of \cite[Theorem~1.3]{pan}.

\begin{lemma}\label{l.F_approx_overlap}
We consider the following setting:
\begin{itemize}
    \item Let $\Gamma$ be a probability measure on a separable Hilbert space $\mathcal{H}$ and let $R:\mathcal{H}\times \mathcal{H} \to \R^{D\times D}$ be a measurable function satisfying $|R|\leq 1$;
    \item Assume that there are centered $\R^D$-valued and real-valued Gaussian processes $(Z(\rho))_{\rho\in\supp\Gamma}$ and $(Y(\rho))_{\rho\in \supp\Gamma}$ with covariances $\E Z(\rho)Z(\rho')^\intercal = \nabla\xi\Ll(R(\rho,\rho')\Rr)$ and $\E Y(\rho)Y(\rho') =\theta\Ll(R(\rho,\rho'\Rr))$.
    \item Write $\la\cdot\ra_\Gamma = \Gamma^{\otimes\infty}$ and define
    \begin{align*}
    \mathcal{P}(\Gamma,R) = \E \log \la \int \exp\Ll(Z(\rho)\cdot \tau - \frac{1}{2}\nabla \xi(R(\rho,\rho))\cdot \tau\tau^\intercal\Rr)\d P_1(\tau) \ra_\Gamma 
    \\
    -\E\log \la \exp\Ll(Y(\rho)-\frac{1}{2}\theta(R(\rho,\rho))\Rr)\ra_\Gamma
\end{align*}
    where $\E$ integrates the Gaussian randomness in $Z(\rho)$ and $Y(\rho)$.
\end{itemize}
Then, for every $\eps>0$, there is a bounded continuous function $\mathcal{F}_\eps: \Ll(\R^{D\times D}\Rr)^{n\times n}\to \R$ for some $n\in\N$ such that
\begin{align*}
    \Ll|\mathcal{P}(\Gamma, R) - \la\mathcal{F}_\eps\Ll(\Ll(R\Ll(\rho^l,\rho^{l'}\Rr)\Rr)_{1\leq l,l'\leq n}\Rr)\ra_\Gamma\Rr|\leq \eps
\end{align*}
uniformly for all $(\Gamma, R)$ as described,
where $\Ll(\rho^l\Rr)_{l\in\N}$ is i.i.d.\ sequence under $\la \cdot\ra_\Gamma$.
\end{lemma}

The pair $(\Gamma, R)$ defines an abstract overlap structure that is relevant here.

\begin{remark}\label{r.P(Gamma,R)}
We show that familiar objects are of the form $\mathcal{P}(\Gamma, R)$:
\begin{itemize}
    \item $A_N(x) = \E\mathcal{P}(\Gamma^x,R_N)$ for $\Gamma^x $ being the distribution of $\frac{1}{\sqrt{N}}\sigma$ under $\la \cdot \ra_x$ and $R_N:(\sigma,\sigma')\mapsto \frac{\sigma\sigma'^\intercal}{N}$, where $\E$ integrates the Gaussian randomness in $\la \cdot\ra_x$;
    \item $\mathscr{P}(\pi) = \E \mathcal{P}(\mathfrak{R}, R)$ for $\mathfrak{R}$ the aforementioned RPC with overlap distributed uniformly on $[0,1]$ and $R:(\alpha,\alpha')\mapsto \pi(\alpha\wedge \alpha')$, where $\E$ integrates the randomness in $\mathfrak{R}$.
\end{itemize}
It is straightforward to check the validity of the first identity. To see the second, one can use the computation in~\eqref{e.logexpy} with $N=1$ to identify the second term in $\E \mathcal{P}(\mathfrak{R}, R)$ with the second term in $\mathscr{P}(\pi)$.
\end{remark}

For a sequence of random arrays $(A_n)_{n\in\N}$ where $A_n = \Ll(A^{l,l'}_n\Rr)_{l,l'\in\N}$, we say that $A_n$ \textit{converges weakly} to some random array $A =\Ll(A^{l,l'}\Rr)_{l,l'\in\N}$ and write
\begin{align*}
    A_n \rightharpoonup A \quad \text{as }n\to\infty
\end{align*}
if, for every $k\in\N$, $A_n^{\leq k}$ converges in distribution to $A^{\leq k}$ as $n\to\infty$.

For any real-valued random variable $X$, its \textit{quantile function} is the left-continuous increasing function $\zeta:[0,1]\to\R$ such that
\begin{align*}
    \E g(X) = \int_0^1 g(\zeta(s))\d s
\end{align*}
for every bounded measurable function $g$. The quantile function can be obtained by taking the left-continuous inverse of the probability distribution function, and vice versa.

\begin{proof}[Proof of Proposition~\ref{p.lower_bd}]
Due to the presence of $N^{-\frac{1}{16}}$, we can verify $\lim_{N\to\infty} \sup_{x} |F_N- F^x_N|=0$. Therefore, it suffices to bound $\liminf_{N\to\infty}\E_x F^x_N$ from below. 
Since
\begin{align*}
    \liminf_{N\to\infty} r_N\geq \liminf_{N\to\infty} (N+1)r_{N+1}-Nr_N
\end{align*}
holds for a sequence of real numbers $(r_N)_{N\in \N}$, we can use Lemma~\ref{l.cavity} to see that
\begin{align*}
    \liminf_{N\to\infty}\E_x F^x_N \geq \liminf_{N\to\infty} \E_x A_N(x).
\end{align*}
We want to choose a sequence of permutation parameters. Let $((f_j,n_j,h_j))_{j\in\N}$ be an enumeration of
\begin{align*}
    \Ll\{(f,n,h):n\in\N;\; h\in \N^3;\; \text{$f$ is an monomial of $\Ll(R^{l,l'}_{i,j}\Rr)_{1\leq i,j\leq D;\; 1\leq l,l'\leq n}$}\Rr\}
\end{align*}
and we set
\begin{align*}
\Delta_N(x) = \E\la \Ll| R^{1,1} - \E\la R^{1,1}\ra_x\Rr|\ra_x +  \sum_{j=1}^\infty 2^{-j} \Delta^x(f_j,n_j,h_j)
\end{align*}
where $\Delta^x(f,n,h)$ is defined in~\eqref{e.Delta^x}. By Proposition~\ref{p.perturbation}, we have that $\lim_{N\to\infty}\E_x \Delta_N(x)=0$. 
Using the same argument in the proof of \cite[Lemma~3.3]{pan}, we can find a sequence $\Ll(x^N\Rr)_{N\in\N}$ such that
\begin{gather}
\lim_{\N\to\infty} \Delta_N\Ll(x^N\Rr)=0, \label{e.Deltato0}
\\
    \liminf_{N\to\infty}\E_x A_N(x) \geq \liminf_{N\to\infty} A_N\Ll(x^N\Rr).\notag
\end{gather}
Hence, it suffices to evaluate $\liminf_{N\to\infty} A_N\Ll(x^N\Rr)$.

Choose an increasing sequence of integers $\Ll(N^{(0)}_k\Rr)_{k\in\N}$ such that
\begin{align}\label{e.liminfA=limA}
    \liminf_{N\to\infty} A_N\Ll(x^N\Rr) = \lim_{k\to\infty} A_{N^{(0)}_k}\Ll(x^{N^{(0)}_k}\Rr).
\end{align}
Let us make the dependence of $R$ on $N$ explicit by writing $R_N =R$ for spins in $\R^{D\times N}$. Since $R^{1,1}_N \in \S^D_+$ and $|R^{1,1}_N|\leq 1$, we can extract a subsequence $\Ll(N^{(1)}_k\Rr)_{k\in\N}$ from $\Ll(N^{(0)}_k\Rr)_{k\in\N}$ along which
$\E \la R^{1,1}_N \ra_{x^N}$ converges to some $q \in \S^D_+$. 
Since each $ R^{l,l'}_N$ is bounded, we can extract from $\Ll(N^{(1)}_k\Rr)_{k\in\N}$ a further subsequence $\Ll(N^{(2)}_k\Rr)_{k\in\N}$ along which 
\begin{align}\label{e.RtoR}
    R_{N^{(2)}_k} \rightharpoonup R_\infty\quad\text{as }k\to\infty
\end{align}
for some random array $R_\infty$. The distribution of $R_{N^{(2)}_k}$ is induced by $\E\la\cdot\ra_{x^{N^{(2)}_k}}$ and the distribution of $R_\infty$ is induced by $\E\la\cdot\ra_{\mathfrak{R}}$ to be explained below. Due to the concentration of the self-overlap implied by~\eqref{e.Deltato0} and the choice of $\Ll(N^{(1)}_k\Rr)_{k\in\N}$, we have that
\begin{align}\label{e.diag_R}
    R_\infty^{l,l} =q,\quad \forall l\in\N.
\end{align}
As a result of~\eqref{e.Deltato0}, $R_\infty$ satisfies the Ghirlanda--Guerra identities. By the synchronization result \cite[Theorem~4]{pan.vec}, there is a Lipschitz function $\Psi:[0,\infty)\to\S^D_+$ satisfying $\Psi(s)\geqpsd \Psi(s')$ for all $s\geq s'$ such that 
\begin{align}\label{e.R=Phi(tr(R))}
    R^{l,l'}_\infty = \Psi\Ll(\tr\Ll(R^{l,l'}_\infty\Rr)\Rr) \quad \text{a.s.}\ \forall l,l'\in\N.
\end{align}

Denote the quantile function of $\tr\Ll(R^{1,2}_\infty\Rr)$ by $\zeta$. Since the Ghirlanda--Guerra identities hold for $\Ll(\tr\Ll(R^{l,l'}_\infty\Rr)\Rr)_{l,l'\in\N}$, the Dovbysh--Sudakov representation (\cite[Theorem~1.7]{pan}) together with the characterization of RPCs by the overlap distribution (\cite[Theorems~2.13 and~2.17]{pan}) implies that $\Ll(\tr\Ll(R^{l,l'}_\infty\Rr)\Rr)_{l\neq l'}$ has the law of a RPC with overlap distribution given by the quantile function $\zeta$. Therefore, we can represent $\Ll(\tr\Ll(R^{l,l'}_\infty\Rr)\Rr)_{l\neq l'}$ by $\tr\Ll(R^{l,l'}_\infty\Rr) = \zeta\Ll(\alpha^l\wedge\alpha^{l'}\Rr)$ for all $l\neq l'$, where $\Ll(\alpha^{l}\Rr)_{l\in\N}$ is sampled from $\mathfrak{R}^{\otimes \infty}$. 

Then, we want to obtain a representation of the entire array $R_\infty$.
First, we look for an upper bound for $\zeta$. 
Note that~\eqref{e.R=Phi(tr(R))} implies $R^{1,2}_\infty \in\S^D_+$. Since the Cauchy--Schwarz inequality implies that $v\cdot R^{1,2}_N v\leq \frac{1}{2} v\cdot \Ll(R^{1,1}_N+R^{2,2}_N\Rr)v$ for every $v\in\R^D$, we can deduce from~\eqref{e.diag_R} that $ R^{1,2}_\infty\leqpsd q$ and thus $\tr\Ll(R^{1,2}_\infty\Rr)\leq \tr(q)$ a.s. Setting $r=\tr(q)$, we get
\begin{align}\label{e.zeta(s)<r}
    \zeta(s)\leq r,\quad\forall s\in[0,1].
\end{align}
Then,~\eqref{e.R=Phi(tr(R))} together with~\eqref{e.diag_R} also implies
\begin{align}\label{e.Phi(r)=q}
    \Psi(r) = q.
\end{align}
Hence, we can represent $R_\infty$ in the following way
\begin{align}\label{e.R^l,l'}
    R^{l,l'}_\infty = \Psi\Ll(\zeta\Ll(\alpha^l\wedge\alpha^{l'}\Rr)+\mathds{1}_{l=l'}\Ll(r-\zeta\Ll(\alpha^l\wedge\alpha^{l'}\Rr)\Rr)\Rr),\quad\forall l,l'\in\N.
\end{align}
So, the distribution of the random array $R_\infty$ is induced by $\E\la\cdot\ra_{\mathfrak{R}}$ where $\la\cdot\ra_\mathfrak{R} = \mathfrak{R}^{\otimes \infty}$ and $\E$ integrates the randomness in $\mathfrak{R}$. 

Due to the possibility that $r>\zeta(1)$, in general, $R_\infty$ is not a pure RPC.
Hence, we introduce an approximation of $R_\infty$ by RPCs.
Choose a sequence $(\zeta_m)_{m\in\N}$ of left-continuous increasing step functions from $[0,1]$ to $[0,r]$ that converges to $\zeta$ as $m\to\infty$ in $L^1([0,1])$. For each $m$, allowed by~\eqref{e.zeta(s)<r}, we can modify $\zeta_m$ to ensure that $\zeta_m (s)=r$ for $s$ in a small neighborhood of $1$.

For each $m$, define $Q_m = \Ll(Q^{l,l'}_m\Rr)_{l,l'\in\N}$ by
\begin{align*}
    Q^{l,l'}_m = \Psi\Ll(\zeta_m\Ll(\alpha^l\wedge \alpha^{l'}\Rr)\Rr), \quad \forall l,l'\in\N,
\end{align*}
where $\Ll(\alpha^l\Rr)_{l\in\N}$ is again sampled from $\mathfrak{R}^{\otimes\infty}$.
By the convergence of $\zeta_m$ and the continuity of RPCs in the overlap distribution (\cite[Theorem 2.17]{pan}), we can deduce that $\Ll(Q^{l,l'}_m\Rr)_{l\neq l'}$ converges weakly to $\Ll(R^{l,l'}_\infty\Rr)_{l\neq l'}$ under $\E\la\cdot\ra_\mathfrak{R}$. Due to $\zeta_m(1)=r$ and~\eqref{e.Phi(r)=q}, we have $Q^{l,l}_m = q =R^{l,l}_\infty$. Therefore, we conclude that, under $\E\la\cdot\ra_\mathfrak{R}$,
\begin{align}\label{e.QtoR}
    Q_m \rightharpoonup R_\infty \quad\text{as }m\to\infty.
\end{align}

Fix any $\eps>0$. Let $\mathcal{F}_\eps$ be given in Lemma~\ref{l.F_approx_overlap} and recall Remark~\ref{r.P(Gamma,R)}. Notice that $\Psi\circ\zeta_m\in\Pi$ and the distribution of $Q_m$ is characterized by $\mathfrak{R}$ and $(\alpha,\alpha')\mapsto \Psi(\zeta_m(\alpha\wedge\alpha'))$, which satisfies the condition of Lemma~\ref{l.F_approx_overlap}.
For any $r,r'\in\R$, we write $r\approx_\eps r'$ if $|r-r'|\leq \eps$. Using these,~\eqref{e.RtoR}, and~\eqref{e.QtoR}, we can find sufficiently large $k$ and $m$ such that
\begin{align*}
    A_{N^{(2)}_k}\Ll(x^{N^{(2)}_k}\Rr) &\approx_\eps  \E \la \mathcal{F}_\eps\Ll(R^{\leq n}_{N^{(2)}_k}\Rr)\ra_{x^{N^{(2)}_k}} \approx_\eps \E \la \mathcal{F}_\eps\Ll(R^{\leq n}_\infty\Rr)\ra_\mathfrak{R}
    \\
    &\approx_\eps \E \la \mathcal{F}_\eps\Ll(Q^{\leq n}_m\Rr)\ra_\mathfrak{R} \approx_\eps \mathscr{P}(\Psi\circ\zeta_m).
\end{align*}
Using~\eqref{e.liminfA=limA} and the fact that $\Ll(N^{(2)}_k\Rr)_{k\in\N}$ is a subsequence of $\Ll(N^{(0)}_k\Rr)_{k\in\N}$, we obtain from the above display that
\begin{align*}
    \liminf_{N\to\infty}A_N\Ll(x^N\Rr)\geq \inf_{\pi\in\Pi}\mathscr{P}(\pi) - 4\eps.
\end{align*}
The desired lower bound follows by sending $\eps\to0$.
\end{proof}

\section{Removing the correction term}
\label{s.remove}

We remove the correction and prove Theorem~\ref{t.remove} by using the Hamilton--Jacobi technique in \cite[Section~5]{mourrat2020extending} which was set in the case $D=1$. For $D\geq 1$, we consider the equation on the cone of positive definite matrices and thus need results from \cite{chen2022hamilton2}.

Recall that we have endowed $\S^D$ with the Frobenius inner product, which induces the natural topology on $\S^D$.
For $N\in\N$ and $(t,x) \in [0,\infty)\times\S^D$, we define
\begin{gather}\label{e.cF_N}
    \cF_N(t,x) = \E \hat \cF_N(t,x),
    \\
    \hat\cF_N(t,x) = \frac{1}{N} \log\int \exp\Ll(H_N(\sigma) - \frac{N}{2}\xi\Ll(\frac{\sigma\sigma^\intercal}{N}\Rr) +tN\xi\Ll(\frac{\sigma\sigma^\intercal}{N}\Rr) + x\cdot \sigma\sigma^\intercal \Rr)\d P_N(\sigma). \notag
\end{gather}
Since the computations in this section only involve the self-overlap, we set $R = \frac{1}{N}\sigma\sigma^\intercal$. We denote the derivatives in $t$ and $h$ by $\partial_t$ and $\nabla$ respectively. 
Here, $\nabla$ is defined with respect to the Frobenius inner product on $\S^D$. 
Recall that we have chosen an orthogonal basis $\{e_i\}_{i=1}^{D(D+1)/2}$ of $\S^D$, and we define the Laplace operator $\Delta = \sum_{i=1}^{D(D+1)/2}(e_i\cdot\nabla)^2$. We often write $\cF_N =\cF_N(t,x)$ for simplicity. 
\begin{lemma}\label{l.cF_N_properties}
Assume that $\xi$ is convex on $\S^D_+$. The following holds:
\begin{itemize}
    \item for each $N$, $\cF_N$ is Lipschitz, convex, and increasing in the sense that $\cF_N(t,x) \leq \cF_N(t',x')$ if $t\leq t'$ and $x\leqpsd x'$;
    \item the Lipschitzness is uniform in $N$, namely, $\sup_{N\in\N}\|\cF_N\|_\mathrm{Lip}<\infty$;
    \item there is a constant $C>0$ such that, everywhere on $(0,\infty)\times \S^D$ and for every $N$,
    \begin{align}\label{e.approx_hj}
        0\leq \partial_t \cF_N - \xi\Ll(\nabla \cF_N\Rr) \leq C\Ll(N^{-1} \Delta \cF_N\Rr)^\frac{1}{2} + C\E \Ll|\nabla\hat \cF_N -\nabla \cF_N\Rr|.
    \end{align}
\end{itemize}
\end{lemma}

The convexity of $\xi$ is only needed for the lower bound in~\eqref{e.approx_hj}. Notice that $\xi$ is only required to be convex on $\S^D_+$ instead of the entire space $\R^{D\times D}$.

\begin{proof}

For $(t,x) \in (0,\infty)\times \S^D$, we can compute that
\begin{align}\label{e.1st_der_FN}
    \partial_t \cF_N =\E\la \xi(R)\ra,  \qquad \nabla \hat\cF_N = \la R \ra, \qquad \nabla \cF_N = \E \la R \ra
\end{align}
and, for any $(s,y)\in\R\times \S^D$,
\begin{align}\label{e.2nd_der_FN}
    \frac{\d^2}{\d \eps^2}\hat\cF_N(t+\eps s, x+\eps y)\,  \Big|_{\eps =  0} = N  \la \Ll(s \xi(R)+y\cdot R\Rr)^2 - \la s \xi(R)+y\cdot R \ra^2 \ra\geq 0. 
\end{align}
Since $|R|\leq 1$ a.s.\ and $\xi$ is locally Lipschitz, we have that $\mathcal{F}_N$ is Lipschitz in both variables with $\sup_{N}\| \mathcal{F}_N\|_\mathrm{Lip}<\infty$.
Since $\xi\geq 0$ on $\S^D_+$ and $R\in \S^D_+$, we get that $(\partial_t,\nabla )\cF_N\in [0,\infty)\times \S^D_+$, which implies that $\cF_N$ is increasing. Recognizing that the second-order derivative is a variance, we deduce the convexity of $\cF_N$. Setting $s=0$ and $y=e_i$ for each~$i$ in~\eqref{e.2nd_der_FN} and summing up, we obtain
\begin{align*}
    \Delta \cF_N = N  \E \la | R|^2 - |\la  R\ra |^2 \ra = N  \E \la | R - \la  R\ra |^2 \ra,
\end{align*}
which together with~\eqref{e.1st_der_FN} and the local Lipschitzness of $\xi$ implies the upper bound in~\eqref{e.approx_hj}. The lower bound in~\eqref{e.approx_hj} follows from~\eqref{e.1st_der_FN}, the convexity of $\xi$ on $\S^D_+$, the observation $R\in\S^D_+$, and Jensen's inequality.
\end{proof}

From~\eqref{e.approx_hj}, $\cF_N$ is expected to be a viscous approximation of the solution $f$ to the equation
\begin{align}\label{e.hj}
\partial_t f- \xi(\nabla f) =0 \quad \text{on $\Ll(0,\infty\Rr)\times \S^D$}. 
\end{align}
We make sense of the solution to~\eqref{e.hj} in the viscosity sense. Let $O,U\subset\S^D$ satisfy $O\subset U$. A function $f:[0,\infty)\times U\to\R$ is a \textit{viscosity subsolution} (respectively, \textit{supersolution}) of
\begin{align*}
    \partial_t f- \xi(\nabla f) =0 \quad \text{on $\Ll(0,\infty\Rr)\times O$}
\end{align*}
if whenever there is a smooth $\phi:(0,\infty)\times O\to\R$ such that $f-\phi$ achieves a local maximum (respectively, minimum) at some $(t,x) \in (0,\infty)\times O$, we have $(\partial_t \phi-\xi(\nabla\phi))(t,x)\leq 0$ (respectively, $\geq 0$). If $f$ is both a viscosity subsolution and supersolution, we call $f$ a \textit{viscosity solution}.

Since $\S^D$ is isometric to $\R^{D(D+1)/2}$ via the orthogonal basis $\{e_i\}$, all classical theory of viscosity solutions are available for~\eqref{e.hj}. For instance, due to the assumption that $\xi$ is locally Lipschitz, \cite[Theorem~1 in Section~10.2]{evans2010partial} ensures the uniqueness of the solution to~\eqref{e.hj} given an initial condition.

Recall $\mathscr{P}(\pi,x)$ defined in~\eqref{e.parisi_functional} and $\mathscr{P}(\pi)=\mathscr{P}(\pi,0)$. 
We denote by $\psi$ the pointwise limit of $\cF_N(0,\cdot)$ (if it exists).
Applying Theorem~\ref{t.main} with $P_1$ replaced by normalized $e^{x\cdot\sigma\sigma^\intercal}\d P_1(\sigma)$, we can get
\begin{align}\label{e.psi=}
    \text{$\xi$ is convex on $\R^{D\times D}$}\quad \implies \quad \psi(x)=\lim_{N\to\infty}\cF_N(0,x) = \inf_{\pi\in\Pi}\mathscr{P}(\pi,x),\quad\forall x\in \S^D.
\end{align}

By the Lipschitzness of $\cF_N$ uniform in $N$ as stated in Lemma~\ref{l.cF_N_properties}, if $\cF_N$ converges pointwise on a dense set, we can upgrade this to the convergence in the local uniform topology, namely, uniform convergence on every compact subset of $[0,\infty)\times \S^D$. Hence, this is the notion of convergence we consider.

\begin{proposition}\label{p.FN_cvg_f}
Assume that $\xi$ is convex on $\S^D_+$. As $N\to\infty$, $\cF_N$ converges in the local uniform topology to the unique viscosity solution $f$ of~\eqref{e.hj} with initial condition $f(0,\cdot)=\psi$ given in~\eqref{e.psi=}.
\end{proposition}
\begin{proof}
Since $\cF_N$ is Lipschitz uniformly in $N$, the Arzel\`a--Ascoli theorem implies that any subsequence of $(\cF_N)_{N\in\N}$ has a further subsequence that converges in the local uniform topology to some $f$. It suffices to show that the subsequential limit $f$ is the viscosity solution. For lighter notation, we assume that the entire sequence $\cF_N$ converges to $f$.

We divide the proof into two parts, verifying that $f$ is a subsolution in the first part and a supersolution in the second part. It is easy to see that replacing ``local extremum'' by ``strict local extremum'' in the definition of viscosity solutions yields an equivalent definition.

\textit{Part~1.}
Let $(t,x)\in(0,\infty)\times \S^D$ and smooth $\phi$ satisfy that $f-\phi$ has a strict local maximum at $(t,x)$. The goal is to show that 
\begin{equation}
\label{e.eq.phi}
 (\partial_t\phi-\xi(\nabla\phi))(t,x)\leq 0.
\end{equation}
By the local uniform convergence, there exists $(t_N,x_N) \in (0,\infty) \times \S^D$ such that $\cF_N - \phi$ has a local maximum at $(t_N,x_N)$, and $\lim_{N\to\infty}(t_N,x_N)=(t,x)$. Notice that 
\begin{equation}  
\label{e.gradient.equalities}
(\partial_t,\nabla)( \cF_N - \phi)(t_N,x_N) = 0 .
\end{equation}
Throughout this proof, we denote by $C < \infty$ a constant that may vary from one occurrence to the next and is allowed to depend on $(t,x)$ and~$\phi$. 

We want to show that, for every $y \in \S^D$ with $|y| \le C^{-1}$,
\begin{equation}
\label{e.hessian.est}
0 \le \cF_N(t_N,x_N + y) - \cF_N(t_N,x_N ) - y \cdot \nabla \cF_N(t_N,x_N) \le C |y|^2.
\end{equation}
The convexity of $\cF_N$ gives the first inequality. To derive the other, we start by using Taylor's expansion:
\begin{multline}  
\label{e.taylor}
\cF_N(t_N,x_N + y) - \cF_N(t_N,x_N ) 
\\= y \cdot \nabla \cF_N(t_N,x_N) + \int_0^1 (1-s) y\cdot \nabla\Ll(y \cdot \nabla \cF_N\Rr)(t_N,x_N + s y) \, \d s.
\end{multline}
The same holds with $\cF_N$ replaced by $\phi$. By the local maximality of $\cF_N - \phi$ at $(t_N,x_N)$,
\begin{equation*}  \cF_N(t_N,x_N + y) - \cF_N(t_N,x_N ) \le \phi(t_N,x_N + y) - \phi(t_N,x_N ) 
\end{equation*}
holds for every $|y| \le C^{-1}$.
The above two displays along with~\eqref{e.gradient.equalities} imply
\begin{equation*}  \int_0^1 (1-s)  y\cdot \nabla\Ll(y \cdot \nabla \cF_N\Rr)(t_N,x_N + s y)  \, \d s \le \int_0^1 (1-s)  y\cdot \nabla\Ll(y \cdot \nabla \phi\Rr)(t_N,x_N + s y)  \, \d s.
\end{equation*}
Since the function $\phi$ is smooth, the right side of this inequality is bounded by $C |y|^2$. Using~\eqref{e.taylor} once more, we obtain~\eqref{e.hessian.est}. 

Next, setting $B = \Ll\{ (t',x') \in [0,\infty) \times \S^D \ : \ |t'-t| \le C^{-1} \text{ and } |x'-x| \le C^{-1} \Rr\}$ and $\delta_N = \E \Ll[ \sup_B \Ll|\hat\cF_N - \cF_N\Rr|  \Rr]$, we show
\begin{equation}
\label{e.concentration}
\E \Ll[ \Ll|\nabla\hat\cF_N - \nabla \cF_N\Rr|(t_N,x_N) \Rr] \le C \delta_N^\frac 1 2.
\end{equation}
Using the convexity of $\hat\cF_N$ in~\eqref{e.2nd_der_FN}, we have
\begin{equation*}  \hat\cF_N(t_N,x_N+ y) \ge\hat\cF_N(t_N,x_N) + y \cdot \nabla\hat\cF_N(t_N,x_N).
\end{equation*}
Combining this with \eqref{e.hessian.est}, we obtain that, for every $|y| \le C^{-1}$,
\begin{equation*}  y \cdot \Ll(\nabla\hat\cF_N - \nabla \cF_N\Rr)(t_N,x_N)\le 2 \sup_B \Ll|\hat\cF_N - \cF_N\Rr| + C |y|^2 .
\end{equation*}
For some deterministic $\lambda \in [0,C^{-1}]$ to be determined, we fix the random matrix
\begin{equation*}  y = \lambda \frac{\Ll(\nabla\hat\cF_N - \nabla \cF_N\Rr)(t_N,x_N)}{\Ll|\nabla\hat\cF_N - \nabla \cF_N\Rr|(t_N,x_N)},
\end{equation*}
to get
\begin{equation*}  \lambda \Ll|\nabla\hat\cF_N - \nabla \cF_N\Rr|(t_N,x_N)  \le 2 \sup_B \Ll|\hat\cF_N - \cF_N\Rr|  + C \lambda^2 .
\end{equation*}
By the standard concentration result (e.g.\ \cite[Theorem~1.2]{pan}) and an $\eps$-net to cover $B$, we can see $\lim_{\N\to\infty} \delta_N =0$. Taking the expectation in the above display and choosing $\lambda = \delta_N^\frac{1}{2}$, we obtain \eqref{e.concentration}. 

Since~\eqref{e.hessian.est} implies that $|\Delta \cF_N(t_N,x_N)|\leq C$, using~\eqref{e.approx_hj}, \eqref{e.gradient.equalities}, and~\eqref{e.concentration}, we arrive at
\begin{align*}
    \Ll(\partial_t\phi - \xi(\nabla\phi)\Rr)(t_N,x_N)\leq CN^{-\frac{1}{2}} + C\delta^\frac{1}{2}_N.
\end{align*}
Sending $N\to\infty$ and using the convergence of $(t_N,x_N)$ to $(t,x)$, we get~\eqref{e.eq.phi}.

\textit{Part~2.}
Let $(t,x)\in(0,\infty)\times \S^D$ and smooth $\phi$ satisfy that $f-\phi$ has a strict local minimum at $(t,x)$.
Since $\cF_N$ converges locally uniformly to $f$, there is a sequence $((t_N,x_N))_{N\in\N}$ such that $\lim_{N\to\infty}(t_N,x_N)=(t,x)$ and $\cF_N-\phi$ has a local minimum at $(t_N,x_N)$. 
In particular, \eqref{e.gradient.equalities} still holds.
Using these and the lower bound in~\eqref{e.approx_hj}, after sending $N\to\infty$, we obtain
\begin{align*}\Ll(\partial_t\phi-\xi\Ll(\nabla \phi\Rr)\Rr)(t,x)\geq 0,
\end{align*}
which verifies that $f$ is a supersolution and completes the proof.
\end{proof}

Next, we want to restrict the equation~\eqref{e.hj} to a smaller set so that the variational formula for the solution optimizes over the smaller set. We denote by $\S^D_{++}$ the set of positive definite matrices, which is the interior of the closed set $\S^D_+$. We consider
\begin{align}\label{e.hj_cone}
    \partial_t f- \xi(\nabla f) =0 \quad \text{on $\Ll(0,\infty\Rr)\times \S^D_{++}$}. 
\end{align}
We state the well-posedness of this equation and variational representations below. Notice that we do not impose any boundary condition on $\partial\S^D_+$. This is possible by only considering increasing solutions.

\begin{proposition}\label{p.hj_on_cone}
For every Lipschitz  $\psi:\S^D_+\to\R$ that is increasing in the sense that $\psi(x)\leq \psi(x')$ if $x\leqpsd x'$, there is a viscosity solution $f:[0,\infty)\times \S^D_+\to \R$ to~\eqref{e.hj_cone} satisfying $f(0,\cdot) =\psi$, which is unique in the class of increasing and Lipschitz functions. Moreover,
\begin{itemize}
    \item if $\xi$ is convex on $\S^D_+$, then $f$ admits the \emph{Hopf--Lax} representation:
    \begin{align}\label{e.Hopf-Lax}
        f(t,x) = \sup_{y\in \S^D_+}\Ll\{\psi(x+y)-t \xi^*\Ll(\frac{y}{t}\Rr)\Rr\},\quad\forall (t,x)\in [0,\infty)\times \S^D_+;
    \end{align}
    \item if $\psi$ is convex, then $f$ admits the \emph{Hopf} representation:
    \begin{align*}
        f(t,x) =\sup_{z\in\S^D_+}\inf_{y\in\S^D_+}\{z\cdot(x-y)+\psi(y)+t\xi(z)\},\quad\forall (t,x)\in [0,\infty)\times \S^D_+.
    \end{align*}
\end{itemize}
\end{proposition}
\noindent In~\eqref{e.Hopf-Lax}, $\xi^*$ is the convex conjugate described in~\eqref{e.xi^*}.

\begin{proof}
This is an extraction of results listed in \cite[Theorem~1.2]{chen2022hamilton2}. Relevant function classes therein are defined in the beginning two paragraphs of \cite[Section~1.1.3]{chen2022hamilton2}. By~\eqref{e.assump_xi}, $\xi$ is increasing on $\S^D_+$, hence satisfying the condition on the nonlinearity of the equation (condition $\mathsf{H}\lfloor_\C \in \Gamma^\nearrow_\mathrm{locLip}$ there; $\mathsf{H}$ and $\C$ there correspond to $\xi$ and $\S^D_+$ here). The existence and uniqueness is in \cite[Theorem~1.2~(2)]{chen2022hamilton2} ($\mathring\C$ corresponds to $\S^D_{++}$ here; uniqueness actually holds in a slightly larger class $\mathfrak{M}\cap\mathfrak{L}_\mathrm{Lip}$). The Lipschitzness of $f$ is in (2a) and the monotonicity follows from (2b) and the main statement of (2) ($f\in \mathfrak{M}$ which is the class of functions increasing in $x$ for each fixed $t$).

Finally, since $\S^D_+$ is a closed convex cone that satisfies the \emph{Fenchel--Moreau} property described in \cite[Definition~6.1]{chen2022hamilton2} which was proved in \cite[Proposition~B.1]{HBJ}, the two representations of the solution are available due to \cite[Theorem~1.2~(2d)]{chen2022hamilton2}.
\end{proof}

\begin{proof}[Proof of Theorem~\ref{t.remove}]
Let $f$ be given by Proposition~\ref{p.FN_cvg_f}. By Lemma~\ref{l.cF_N_properties}, $f$ is also Lipschitz, convex, and increasing and so is $\psi=f(0,\cdot)$ given in~\eqref{e.psi=}. Since $f$ is the viscosity solution of~\eqref{e.hj}, it follows from the definition that $f$ is a viscosity solution of~\eqref{e.hj_cone}. The uniqueness of $f$ follows from Proposition~\ref{p.hj_on_cone}.
Notice that $\cF_N(\frac{1}{2},0) = \frac{1}{N}\E\log\int\exp H_N(\sigma)\d P_N(\sigma)$. Hence, due to the convexity of $\xi$ and $\psi$, the limit of the original free energy is given by the Hopf--Lax and Hopf representations evaluated at $(t,x) = (\frac{1}{2},0)$.
\end{proof}

\section{Enriched models}\label{s.enriched}

Recently, Mourrat initiated a PDE approach to spin glasses \cite{mourrat2021hamilton,mourrat2022parisi,mourrat2020extending,mourrat2021nonconvex,mourrat2023free} (similar considerations also appeared in physics literature \cite{guerra2001sum, genovese2009mechanical, barra2010replica, barra2013mean}). Free energy enriched by a RPC as the additional field is recast as the solution to a Hamilton--Jacobi equation. In this section, we prove the Hopf--Lax representation of the limit free energy for vector spins, Theorem~\ref{t.enriched}, which extends the results in \cite{mourrat2022parisi,mourrat2020extending}.

We start by describing the enriched model. Recall that we write $\sigma=(\sigma_1,\dots,\sigma_N)\in\R^{D\times N}$ where $\sigma_i \in \R^D$ is the $i$-th column vector of $\sigma$. For each $i\in\N$, let $(w^\pi_i(\alpha))_{\alpha\in\supp\mathfrak{R}}$ be an independent copy of $(w^\pi(\alpha))_{\alpha\in\supp\mathfrak{R}}$ (see~\eqref{e.w^pi}) conditioned on $\mathfrak{R}$.
For $(t, \mu) \in [0,\infty)\times \Pi$, we consider
\begin{gather*}
    H^{t,\mu}(\sigma,\alpha) = \sqrt{t}H_N(\sigma) - \frac{t}{2}N\xi\Ll(\frac{\sigma\sigma^\intercal}{N}\Rr) + \sum_{i=1}^N w^\mu_i(\alpha)\cdot \sigma_i  -\frac{1}{2}\mu(1)\cdot\sigma\sigma^\intercal,
    \\
    F_N(t,\mu) = \frac{1}{N}\E\log\iint \exp\Ll( H^{t,\mu}(\sigma,\alpha) \Rr)\d P_N(\sigma)\d \mathfrak{R}(\alpha).
\end{gather*}
To describe the limit, we set
\begin{align*}
    \psi(\mu) = \E\log\iint \exp\Ll(w^\mu\cdot \tau - \frac{1}{2}\mu(1)\cdot\tau\tau^\intercal\Rr)\d P_1(\tau)\d \mathfrak{R}(\alpha)
\end{align*}
which is equal to $F_N(0,\mu)$ for every $N$.
Recall $\xi^*$ defined in \eqref{e.xi^*}.
\begin{theorem}\label{t.enriched}
If $\xi$ is convex, then for every $(t,\mu)\in[0,\infty)\times \Pi$,
\begin{align*}
    \lim_{N\to\infty} F_N(t,\mu) = \inf_{\pi\in\Pi} \Ll\{\psi(\mu+ t\nabla\xi\circ\pi) + \frac{t}{2}\int_0^1\xi^*\Ll(\nabla\xi(\pi(s))\Rr)\d s\Rr\}.
\end{align*}
\end{theorem}
We explain the relationship between the functional inside the infimum and $\mathscr{P}(\pi)$ in \eqref{e.parisi_functional}.
Recall the definition of $\theta$ in \eqref{e.theta}. The convexity of $\xi$ implies $\xi^*(\nabla\xi(\cdot))=\theta$. Therefore, we have $\mathscr{P}(\pi) = \psi(\nabla\xi\circ\pi)+\frac{1}{2}\int_0^1\xi^*\Ll(\nabla\xi(\pi(s))\Rr)\d s$ as expected, since $F_N(1,0) = F_N$.

When $D=1$, in the mixed $p$-spin case, $\nabla\xi$ maps $\S^1_+=[0,\infty)$ back to itself. Hence, writing $\nu = t\nabla\xi\circ\pi$, we can rewrite the formula in the above theorem as
\begin{align*}
    \inf_{\nu\in\Pi}\Ll\{\psi(\mu+\nu)+\frac{t}{2}\int_0^1\xi^*(\nu(s))\d s\Rr\}
\end{align*}
which recovers \cite[Theorem~1.1]{mourrat2020extending} (actually, the above optimizes over a smaller set). The proof in \cite{mourrat2020extending} is based on verifying the Parisi formula for the enriched free energy without self-overlap correction (\cite[Section~4]{mourrat2020extending}) and then transforming the Parisi formula into the Hopf--Lax formula. Using our arguments, we can offer a direct proof. 
To avoid repetition, we only sketch key steps.

\begin{proof}[Sketch of the proof of Theorem~\ref{t.enriched}]
We sketch key steps for the upper bound, the perturbation, and the lower bound. The details for the latter two can be seen in~\cite[Corollary 6.11]{HJ_critical_pts} (also see~\cite{chen2022pde} for the case $D=1$).

\textit{Upper bound.}
For $r\in[0,1]$, we set
\begin{align*}
    H^r_N(\sigma,\alpha) = H^{rt,\mu}_N(\sigma,\alpha)+\sqrt{1-r}\sum_{i=1}^N w^{t\nabla\xi\circ\pi}_i(\alpha)\cdot \sigma_i - \frac{1-r}{2}t\nabla\xi\circ\pi(1)\cdot\sigma\sigma^\intercal
    \\+\sqrt{rNt}y(\alpha)-\frac{rN}{2}t\theta(\pi(1))
\end{align*}
similar to~\eqref{e.H^t_upper_bd} with $(y(\alpha))_\alpha$ defined above~\eqref{e.H^t_upper_bd}. Then, we define $\varphi$ as in~\eqref{e.phi(r)} but for $H^r_N(\sigma,\alpha)$ in the above. Then, we can compute the derivative of $\varphi$ to get a result similar to~\eqref{e.dphi(r)}. We can then follow the computation below~\eqref{e.dphi(r)} to verify the upper bound.

\textit{Perturbation.}
In addition to $\sigma$, we need to add terms in the perturbation for $\alpha$.
Again, let $(a_n)_{n\in\N}$ enumerate $\{q\in\S^D_+:|q|\leq 1\}\cap \Q^{D\times D}$. Additionally, let $(\lambda_n)_{n\in\N}$ enumerate $[0,1]\cap \Q$.
For $N\in\N$ and $h\in \N^4$, let $(H^h_N(\sigma,\alpha))_{\sigma\in \R^{D\times N},\,\alpha\in\supp\mathfrak{R}}$ be an independent centered Gaussian process with variance
\begin{align*}\E H^h_N(\sigma,\alpha)H^h_N(\sigma',\alpha') = N\Ll(a_{h_1}\cdot \Ll(\frac{\sigma\sigma'^\intercal}{N}\Rr)^{\odot h_2}+\lambda_{h_3}\alpha\wedge\alpha'\Rr)^{h_4}.
\end{align*}
Recall $g_i(\sigma)$ for $i\in\{1,\dots, D(D+1)/2\}$ introduced above~\eqref{e.H^pert}. For $x=((x_h),(x_i))\in [0,3]^{\N^4}\times[0,3]^{ \frac{D(D+1)}{2}}$, we set
\begin{align*}
    H^{\pert, x}_N(\sigma,\alpha) = \sum_{h\in\N^4}x_hc_h H^h_N(\sigma,\alpha)+\sum_{i=1}^{D(D+1)/2}x_ig_i(\sigma)
\end{align*}
where $c_h$ is a positive constant chosen sufficiently small so that the variance of $c_h H^h_N(\sigma)$ is bounded by $2^{-2|h|}$ for $|h|=h_1+h_2+h_3+h_4$.

For $l,l'\in\N$, we set $R^{l,l'}_\sigma = \frac{\sigma^l(\sigma^{l'})^\intercal}{N}$ and $R^{l,l'}_\alpha = \alpha^l\wedge\alpha^{l'}$. The entire overlap $R^{l,l'} = \Ll(R^{l,l'}_\sigma,\,R^{l,l'}_\alpha\Rr)$ is now $\R^{D\times D}\times \R$-valued. Set $R^{l,l'}_h = \Ll(a_{h_1}\cdot \Ll(R^{l,l'}_\sigma\Rr)^{\odot h_2}+\lambda_{h_3}R^{l,l'}_\alpha\Rr)^{h_4}$ for each $h$ and $R^{\leq n} =\Ll(R^{l,l'}\Rr)_{l,l'\leq n}$ for each $n$. Then, the concentration of $R^{l,l}$ and the Ghirlanda--Guerra identities still hold asymptotically as stated in Proposition~\ref{p.perturbation}.

\textit{Lower bound.}
Let $F^x_N(t,\mu)$ be the free energy associated with $H^{t,\mu}_N + N^{-\frac{1}{16}}H^{\pert,x}_N$. For $(Z(\sigma))_\sigma$ and $(Y(\sigma))_\sigma$ defined above~\eqref{e.A_N(x)}, we have that, up to an $\mathcal{O}(N^{-1})$ error and uniformly in $x$, $(N+1)F^x_{N+1}(t,\mu) - NF^x_N(t,\mu)$ is equal to
\begin{align*}
    & A_N(x) 
    \\
    &= \E \log\la \int \exp\Ll(\Ll(\sqrt{t}Z(\sigma)+w^\mu(\alpha)\Rr)\cdot \tau - \frac{1}{2}\Ll(t\nabla\xi\Ll(\frac{\sigma\sigma^\intercal}{N}\Rr)+\mu(1)\Rr)\cdot \tau\tau^\intercal \Rr)  \d P_1(\tau)\ra_x  \notag
    \\
    &\quad - \E \log \la \exp\Ll(\sqrt{t}Y(\sigma)
    - \frac{1}{2}t\theta\Ll(\frac{\sigma\sigma^\intercal}{N}\Rr)\Rr) \ra_x.
\end{align*}
Compared with the previous cavity computation~\eqref{e.A_N(x)}, the additional term is $ w^\mu(\alpha)\cdot \tau -\frac{1}{2}\mu(1)\cdot\sigma\sigma^\intercal$ and rescaling in $t$.

Passing to a subsequence, we can assume that $\Ll(R^{l,l'}\Rr)_{l,l'\in\N}$ converges weakly to some $\Ll(R^{l,l'}_\infty\Rr)_{l,l'\in\N}$ which satisfies the Ghirlanda--Guerra identities. We write $R^{l,l'}_\infty=\Ll(R^{l,l'}_{\infty,\sigma},R^{l,l'}_{\infty,\alpha}\Rr)$ where $\sigma,\alpha$ are purely symbolic. Let $\zeta$ be the quantile function of $\tr\Ll(R^{1,2}_{\infty,\sigma}\Rr) + R^{1,2}_{\infty,\alpha}$. By synchronization, there is $(\Psi_\sigma,\Psi_\alpha):[0,\infty)\to \S^D_+\times [0,1]$ such that $R^{l,l'}_\infty$ has the law of $(\Psi_\sigma(\zeta(s)),\Psi_\alpha(\zeta(s)))$ for $s$ uniformly distributed over $[0,1]$. By the invariance of the RPC (\cite[Theorem~4.4]{pan} and also \cite[Proposition 4.8]{HJ_critical_pts}), since $\alpha^1\cdot\alpha^2$ has uniform law over $[0,1]$, we have that $\Psi_\alpha(\zeta(s))=s$. Therefore, we can represent $R_\infty$ as follows (comparing this with~\eqref{e.R^l,l'}): for every $l,l'\in\N$,
\begin{align*}
    R^{l,l'}_{\infty,\sigma} = \Psi_\sigma\Ll(\zeta\Ll(\alpha^l\wedge\alpha^{l'}\Rr)+\mathds{1}_{l=l'}\Ll(r-\zeta\Ll(\alpha^l\wedge\alpha^{l'}\Rr)\Rr)\Rr),\qquad R^{l,l'}_{\infty,\alpha}=\alpha^l\wedge\alpha^{l'},
\end{align*}
where $r\in\R$ satisfies $r\geq\zeta$ and $(\alpha^l)_{l\in\N}$ is sampled from $\mathfrak{R}^{\otimes\infty}$. We can approximate $R_\infty$ by
\begin{align*}
    Q^{l,l'}_m = \Ll(\pi_m\Ll(\alpha^l\wedge\alpha^{l'}\Rr),\ \alpha^l\wedge\alpha^{l'}\Rr)
\end{align*}
for some sequence $(\pi_m)_{m\in\N}$ in $\Pi$.
Then, we can carry out the same procedure after~\eqref{e.QtoR} to see that the $\liminf$ of $A_N$ is approximately bounded from below by
\begin{align*}
    \E \log\la \int \exp\Ll( \Ll(\sqrt{t}Z(\alpha)+w^\mu(\alpha)\Rr)\cdot \tau - \frac{1}{2}\Ll(t\nabla\xi\circ\pi_m\Ll(1\Rr)+\mu(1)\Rr)\cdot \tau\tau^\intercal \Rr)  \d P_1(\tau)\ra_x  \notag
    \\
    - \E \log \la \exp\Ll( \sqrt{t}Y(\alpha)
    - \frac{1}{2}t\theta\circ\pi_m(1) \Rr)\ra_x
\end{align*}
as $m\to\infty$, where $( Z(\alpha))_\alpha$ and $( Y(\alpha))_\alpha$ are centered Gaussian processes with variances $\E Z(\alpha)Z(\alpha')^\intercal = \nabla\xi\circ \pi_m(\alpha\wedge\alpha')$ and $\E Y(\alpha)Y(\alpha') = \theta\circ\pi_m(\alpha\wedge\alpha')$. Note that $\sqrt{t}Z(\alpha)+w^\mu(\alpha)$ has the same law of $w^{\mu+ t\nabla\xi\circ\pi_m}$. This gives the matching lower bound.
\end{proof}

\small
\bibliographystyle{abbrv}

\begin{thebibliography}{10}

\bibitem{aizenman2003extended}
M.~Aizenman, R.~Sims, and S.~L. Starr.
\newblock {Extended variational principle for the Sherrington--Kirkpatrick
  spin-glass model}.
\newblock {\em Physical Review B}, 68(21):214403, 2003.

\bibitem{barra2013mean}
A.~Barra, G.~Dal~Ferraro, and D.~Tantari.
\newblock Mean field spin glasses treated with {PDE} techniques.
\newblock {\em The European Physical Journal B}, 86(7):1--10, 2013.

\bibitem{barra2010replica}
A.~Barra, A.~Di~Biasio, and F.~Guerra.
\newblock Replica symmetry breaking in mean-field spin glasses through the
  {Hamilton--Jacobi} technique.
\newblock {\em Journal of Statistical Mechanics: Theory and Experiment},
  2010(09):P09006, 2010.

\bibitem{bates2022free}
E.~Bates and Y.~Sohn.
\newblock Free energy in multi-species mixed p-spin spherical models.
\newblock {\em Electronic Journal of Probability}, 27:1--75, 2022.

\bibitem{bates2023parisi}
E.~Bates and Y.~Sohn.
\newblock {Parisi formula for balanced Potts spin glass}.
\newblock {\em arXiv preprint arXiv:2310.06745}, 2023.

\bibitem{chen2022pde}
H.-B. Chen.
\newblock {A PDE perspective on the Aizenman-Sims-Starr scheme}.
\newblock {\em arXiv preprint arXiv:2212.09542}, 2022.

\bibitem{chen2023on}
H.-B. Chen.
\newblock {On the self-overlap in vector spin glasses}.
\newblock {\em arXiv preprint arXiv:2311.09880}, 2023.

\bibitem{HJ_critical_pts}
H.-B. Chen and J.-C. Mourrat.
\newblock On the free energy of vector spin glasses with non-convex
  interactions.
\newblock {\em arXiv preprint arXiv:2311.08980}, 2023.

\bibitem{chen2022hamilton}
H.-B. Chen and J.~Xia.
\newblock {Hamilton--Jacobi equations from mean-field spin glasses}.
\newblock {\em arXiv preprint arXiv:2201.12732}, 2022.

\bibitem{chen2022hamilton2}
H.-B. Chen and J.~Xia.
\newblock {Hamilton--Jacobi equations with monotone nonlinearities on convex
  cones}.
\newblock {\em arXiv preprint arXiv:2206.12537}, 2022.

\bibitem{HBJ}
H.-B. Chen and J.~Xia.
\newblock {Hamilton–Jacobi equations for inference of matrix tensor
  products}.
\newblock {\em Annales de l'Institut Henri Poincaré, Probabilités et
  Statistiques}, 58(2):755 -- 793, 2022.

\bibitem{chen2013aizenman}
W.-K. Chen.
\newblock {The Aizenman-Sims-Starr scheme and Parisi formula for mixed p-spin
  spherical models}.
\newblock {\em Electronic Journal of Probability}, 18:1--14, 2013.

\bibitem{evans2010partial}
L.~C. Evans.
\newblock {\em Partial Differential Equations}, volume~19.
\newblock American Mathematical Soc., 2010.

\bibitem{genovese2009mechanical}
G.~Genovese and A.~Barra.
\newblock A mechanical approach to mean field spin models.
\newblock {\em Journal of Mathematical Physics}, 50(5):053303, 2009.

\bibitem{guerra2001sum}
F.~Guerra.
\newblock Sum rules for the free energy in the mean field spin glass model.
\newblock {\em Fields Institute Communications}, 30(11), 2001.

\bibitem{gue03}
F.~Guerra.
\newblock Broken replica symmetry bounds in the mean field spin glass model.
\newblock {\em Comm. Math. Phys.}, 233(1):1--12, 2003.

\bibitem{mourrat2021hamilton}
J.-C. Mourrat.
\newblock Hamilton-{J}acobi equations for mean-field disordered systems.
\newblock {\em Ann. H. Lebesgue}, 4:453--484, 2021.

\bibitem{mourrat2021nonconvex}
J.-C. Mourrat.
\newblock Nonconvex interactions in mean-field spin glasses.
\newblock {\em Probab. Math. Phys.}, 2(2):281--339, 2021.

\bibitem{mourrat2022parisi}
J.-C. Mourrat.
\newblock {The Parisi formula is a Hamilton--Jacobi equation in Wasserstein
  space}.
\newblock {\em Canadian Journal of Mathematics}, 74(3):607--629, 2022.

\bibitem{mourrat2023free}
J.-C. Mourrat.
\newblock Free energy upper bound for mean-field vector spin glasses.
\newblock {\em Ann. Inst. Henri Poincar\'{e} Probab. Stat.}, 59(3):1143--1182,
  2023.

\bibitem{mourrat2020extending}
J.-C. Mourrat and D.~Panchenko.
\newblock Extending the {P}arisi formula along a {H}amilton-{J}acobi equation.
\newblock {\em Electron. J. Probab.}, 25:Paper No. 23, 17, 2020.

\bibitem{pan05}
D.~Panchenko.
\newblock Free energy in the generalized {S}herrington--{K}irkpatrick mean
  field model.
\newblock {\em Rev. Math. Phys.}, 17(7):793--857, 2005.

\bibitem{pan.aom}
D.~Panchenko.
\newblock The {P}arisi ultrametricity conjecture.
\newblock {\em Ann. of Math. (2)}, 177(1):383--393, 2013.

\bibitem{pan}
D.~Panchenko.
\newblock {\em The {S}herrington--{K}irkpatrick {M}odel}.
\newblock Springer Monographs in Mathematics. Springer, New York, 2013.

\bibitem{panchenko2014parisi}
D.~Panchenko.
\newblock {The Parisi formula for mixed $ p $-spin models}.
\newblock {\em The Annals of Probability}, 42(3):946--958, 2014.

\bibitem{pan.multi}
D.~Panchenko.
\newblock The free energy in a multi-species {S}herrington--{K}irkpatrick
  model.
\newblock {\em Ann. Probab.}, 43(6):3494--3513, 2015.

\bibitem{pan.potts}
D.~Panchenko.
\newblock Free energy in the {P}otts spin glass.
\newblock {\em Ann. Probab.}, 46(2):829--864, 2018.

\bibitem{pan.vec}
D.~Panchenko.
\newblock Free energy in the mixed {$p$}-spin models with vector spins.
\newblock {\em Ann. Probab.}, 46(2):865--896, \noop{2019}2018.

\bibitem{parisi79}
G.~Parisi.
\newblock Infinite number of order parameters for spin-glasses.
\newblock {\em Phys.\ Rev.\ Lett.}, 43(23):1754, 1979.

\bibitem{parisi80}
G.~Parisi.
\newblock A sequence of approximated solutions to the {SK} model for spin
  glasses.
\newblock {\em J. Phys. A}, 13(4):L115--L121, 1980.

\bibitem{ruelle1987mathematical}
D.~Ruelle.
\newblock {A mathematical reformulation of Derrida's REM and GREM}.
\newblock {\em Communications in Mathematical Physics}, 108:225--239, 1987.

\bibitem{tal.sph}
M.~Talagrand.
\newblock Free energy of the spherical mean field model.
\newblock {\em Probab. Theory Related Fields}, 134(3):339--382, 2006.

\bibitem{Tpaper}
M.~Talagrand.
\newblock The {P}arisi formula.
\newblock {\em Ann. of Math. (2)}, 163(1):221--263, 2006.

\end{thebibliography}
\newcommand{\noop}[1]{} \def\cprime{$'$}

\end{document}